\documentclass{amsproc}

\usepackage{amsmath}
\usepackage{amssymb}
\usepackage{rotating}
\usepackage{graphicx}
\usepackage[tableposition=below]{caption}
\usepackage{bm}
\usepackage{color}

\DeclareGraphicsExtensions{.png,.pdf,.eps}
\usepackage[all,2cell,cmtip]{xy}
\usepackage{tikz}
\usepackage{longtable}
\newtheorem{theorem}{Theorem}[section]
\newtheorem{lemma}[theorem]{Lemma}
\newtheorem{corollary}[theorem]{Corollary}
\newtheorem{proposition}[theorem]{Proposition}

\newtheorem{claim}{Claim}[theorem]

\theoremstyle{definition}

\newtheorem{definition}[theorem]{Definition}

\newtheorem{remark}[theorem]{Remark}

\def\P{{\mathbb P}}

\def\Z{{\mathbb Z}}

\def\cO{{\mathcal{O}}}

\def\cOperatorname#1{\mathop{\rm #1}\nolimits}

\def\Pic{\cOperatorname{Pic}}

\def\rk{\cOperatorname{rk}}

\def\deg{\cOperatorname{deg}}

\def\det{\cOperatorname{det}}

\def\ME{{\cOperatorname{ME}}}

\newcommand{\cME}[1]{\cOverline{\ME}}

\title[Ulrich exterior powers of tangent bundles]
{Varieties with Ulrich exterior powers of the tangent bundle}
\author{Yuta Takahashi and Kiwamu Watanabe}
\address{Department of Mathematics, Faculty of Science and Engineering, Chuo University.
1-13-27 Kasuga, Bunkyo-ku, Tokyo 112-8551, Japan}
\email{yuta0630takahashi0302@gmail.com}
\address{Department of Mathematics, Faculty of Fundamental Science and Engineering, Chuo University.
1-13-27 Kasuga, Bunkyo-ku, Tokyo 112-8551, Japan}
\email{watanabe@math.chuo-u.ac.jp}
\thanks{The first author is partially supported by JST SPRING, Japan Grant Number JPMJSP2170.}
\thanks{The second author is partially supported by JSPS KAKENHI Grant Number 25K06940.}

\begin{document}

\begin{abstract}
We study smooth polarized projective varieties $(X,H)$ whose exterior powers of the tangent bundle are Ulrich. We prove that if $\bigwedge^rT_X$ is $H$-Ulrich for some $0<r<\dim X$, then $X$ is Fano and the intersection number $(-K_X)\cdot H^{n-1}$ is determined explicitly. We then classify the Picard number one case: the only example is the Veronese surface $(\mathbb P^2,\mathcal O_{\mathbb P^2}(2))$.
\end{abstract}

\maketitle

\section{Introduction}
Let $X$ be a smooth projective variety over $\mathbb C$, and let $H$ be a very ample divisor on $X$. A vector bundle $E$ on $(X,H)$ is called $H$-Ulrich if
$$
 H^q(X,E(-iH))=0
 \qquad\text{for every }1\leq i\leq \dim X \text{ and every } q.
$$
The terminology comes from commutative algebra. For the relation between this cohomological definition and Ulrich modules over homogeneous coordinate rings, see for instance \cite[§2.1, §3.2]{CMP21}. Ulrich bundles form a distinguished class of arithmetically Cohen--Macaulay bundles: they have the most linear possible behavior with respect to the embedding and are closely related to syzygies and Chow forms. What will be used below is that they are globally generated and $\mu_H$-semistable; in particular, they are nef. For further background on Ulrich bundles, see \cite{Bea18,Cos17,CMP21,ESW03}.

A central existence problem asks whether a given polarized variety carries an Ulrich bundle. A complementary and much more rigid problem is to determine when a vector bundle canonically attached to the variety is Ulrich. Benedetti, Montero, Prieto-Monta\~nez and Troncoso proved that the only smooth polarized varieties whose tangent bundle is Ulrich are the twisted cubic $(\mathbb P^1,\mathcal O_{\mathbb P^1}(3))$ and the Veronese surface $(\mathbb P^2,\mathcal O_{\mathbb P^2}(2))$ \cite{BMPT23}. Lopez and Raychaudhury subsequently studied twists $T_X(kH)$ and obtained numerical bounds and classification results for several values of $k$ \cite{LR24}. The present paper extends this rigidity problem from the tangent bundle to all intermediate exterior powers $\bigwedge^rT_X$.

This passage is not formal. For $r=1$, the global generation of $T_X$ forces $X$ to be homogeneous, a decisive point in \cite{BMPT23}. For $r>1$, the global generation or nefness of $\bigwedge^rT_X$ does not imply the corresponding property for $T_X$, so the homogeneous-space argument is no longer available. This distinction is reflected in the structure theory. If $T_X$ is nef, the theorem of Demailly--Peternell--Schneider describes the Albanese morphism after a finite \'{e}tale cover, and the Campana--Peternell conjecture predicts that the Fano varieties arising in this setting are rational homogeneous \cite{DPS94,CP91}. Nefness of $\bigwedge^rT_X$ is substantially weaker: a theorem of the second author still produces, after a finite \'{e}tale cover, a Fano fiber space over an Abelian variety, but allows a much broader class of varieties \cite{Wat23}. The question is therefore whether the stronger Ulrich condition restores rigidity in this wider setting.

Our first main theorem gives a structural answer without any assumption on the Picard number.

\begin{theorem}[$=$Theorem~\ref{thm2}]
Let $X$ be a smooth projective variety of dimension $n\geq2$, and let $H$ be a very ample divisor on $X$. If $\bigwedge^rT_X$ is $H$-Ulrich for some $0<r<n$, then $X$ is Fano and
$$
 (-K_X)\cdot H^{n-1}=\frac{n(n+1)}{n+2r}H^n.
$$
\end{theorem}

The Fano conclusion is stronger than the positivity of the displayed intersection number. The key point is that the structure theorem of \cite{Wat23} allows an Abelian base, whereas the $\mu_H$-semistability of the Ulrich bundle $\bigwedge^rT_X$, together with the filtration induced by the relative tangent sequence, rules out a positive-dimensional Abelian base. Thus the Ulrich condition eliminates precisely the non-Fano part permitted by nefness alone. This slope-theoretic reduction is the first essential difference from the tangent-bundle case.

Under the additional assumption $\rho(X)=1$, we obtain a complete classification.

\begin{theorem}[$=$Theorem~\ref{main}]
Let $X$ be a smooth projective variety of dimension $n\geq 2$ with $\rho(X)=1$, and let $H$ be a very ample divisor on $X$. Assume that $\bigwedge^r T_X$ is $H$-Ulrich for some $0<r<n$. Then
$$
 (X,H) \simeq \left(\mathbb P^2,\mathcal O_{\mathbb P^2}(2)\right).
$$
Conversely, $T_{\mathbb P^2}$ is $\mathcal O_{\mathbb P^2}(2)$-Ulrich.
\end{theorem}

The second theorem treats every intermediate exterior power uniformly. For $r=1$ and $n\geq2$, it recovers the Picard number one part of \cite{BMPT23}; for arbitrary $r$, it shows that the weaker positivity of an exterior power yields no new Ulrich examples. The proof is not obtained by reducing to $r=1$. By the first theorem, $X$ is Fano. Writing
$$
 \Pic(X)=\mathbb Z[L],\qquad -K_X=i_XL,\qquad H=aL,
$$
the Chern class formula and the Kobayashi--Ochiai bound force $a=1$ or $a=2$. The decisive new ingredient is the Hodge-theoretic non-vanishing
$$
 H^p(X,\Omega_X^p)\neq0
 \qquad (0\leq p\leq n),
$$
which, through $\bigwedge^rT_X\simeq\Omega_X^{n-r}\otimes\mathcal O_X(-K_X)$, gives an obstruction independent of $r$. It excludes $a=1$ without a case-by-case classification by coindex. For $a=2$, parity and index restrictions, together with the Hilbert polynomial, Kodaira vanishing, and global generation of suitable twisted differential forms, exclude $r\geq2$. Hence $r=1$ and $n=2$. This Hodge-theoretic mechanism is the second essential difference from the proof for the tangent bundle.

We also prove related non-existence results for twists of cotangent bundles, restricted ambient cotangent bundles, exterior powers of the restricted ambient tangent bundle, and normal bundles. These results are independent of the Picard number one classification and further illustrate the rigidity of the Ulrich condition for natural differential and ambient bundles.

The paper is organized as follows. Section~2 collects the preliminary results on Ulrich bundles, exterior powers, nef tangent bundles, and Hodge-theoretic non-vanishing. Section~3 proves the Fano reduction theorem and establishes numerical restrictions for $\bigwedge^rT_X\otimes M$, where $M$ is a line bundle on $X$. Section~4 treats twisted exterior powers in Picard number one, and Section~5 proves the classification theorem and verifies the Veronese surface example. Finally, Section~6 contains the independent non-existence results for cotangent, ambient, and normal bundles, together with a comparison with known examples on Enriques surfaces.

\section*{Notation and conventions}
Throughout this paper, we work over the field of complex numbers $\mathbb C$. Unless otherwise stated, $X$ denotes a smooth projective variety of dimension $n$, and $H$ denotes a very ample divisor on $X$. We use standard notation as in \cite{Har77,CMP21}.
\begin{itemize}
\item For a coherent sheaf $F$ on $X$ and an integer $t$, we write $F(tH):=F\otimes \mathcal O_X(tH)$.
\item We denote by $T_X$ and $\Omega_X^1$ the tangent bundle and the cotangent bundle of $X$, respectively, and we write $\Omega_X^p:=\bigwedge^p\Omega_X^1$.
\item If $H$ defines an embedding $\iota\colon X\hookrightarrow\mathbb P^N$, we denote by $N_X=N_{X/\mathbb P^N}$ the normal bundle.
\item The Picard number of $X$ is denoted by $\rho(X)$.
\item For a vector bundle or a torsion-free sheaf $E$ on $X$, its rank is $\operatorname{rk}(E)$ and its $i$-th Chern class is $c_i(E)$.
\item We denote by $A^1(X)$ the codimension-one Chow group of $X$. Since $X$ is smooth, divisor classes and first Chern classes are regarded as elements of $A^1(X)$. When $A^1(X)\simeq \mathbb Z[H]$, we identify $c_1(E)$ with the integer $c_1$ determined by $c_1(E)=c_1H$ in $A^1(X)$.
\item Let $E$ be a torsion-free sheaf on $X$. The $H$-slope of $E$ is
$$
\mu_H(E):=\frac{c_1(E)\cdot H^{n-1}}{\operatorname{rk}(E)}.
$$
The sheaf $E$ is $\mu_H$-semistable if, for every nonzero proper subsheaf $F\subset E$, one has $\mu_H(F)\leq \mu_H(E)$.
\item A vector bundle $E$ on $X$ is nef if the tautological line bundle $\mathcal O_{\mathbb P(E)}(1)$ is nef.
\end{itemize}

\section{Preliminaries}
We collect the standard facts used later.  They are grouped according to their role in the proof: first the basic properties of Ulrich bundles, then $\mu_H$-semistability and exterior powers, and finally the two geometric inputs, namely a structure theorem of varieties with nef exterior powers of the tangent bundle and a simple Hodge-theoretic non-vanishing statement.

\begin{definition}\label{def:ulrich}
Let $X$ be a smooth projective variety of dimension $n$, and let $H$ be a very ample divisor on $X$. A vector bundle $E$ on $X$ is called $H$-Ulrich if
$$
  H^q(X,E(-iH))=0
$$
for every $1\leq i\leq n$ and every $q$.
\end{definition}

The following standard properties explain why Ulrich bundles are useful in this paper: they provide both positivity and numerical restrictions.

\begin{proposition}\label{prop:c1}
Let $X$ be a smooth projective variety of dimension $n$, let $H$ be a very ample divisor on $X$, and let $E$ be an $H$-Ulrich vector bundle on $X$. Then the following hold.
\begin{enumerate}
\item $E$ is globally generated; in particular, $E$ is nef.
\item $E$ is $\mu_H$-semistable.
\item The first Chern class satisfies
$$
c_1(E)\cdot H^{n-1}=\frac{\operatorname{rk}(E)}{2}\left(K_X+(n+1)H\right)\cdot H^{n-1}.
$$
\end{enumerate}
\end{proposition}

\begin{proof}
The global generation and $\mu_H$-semistability are standard consequences of the Ulrich condition; see \cite[Corollary~3.2.10 and Proposition~3.3.14]{CMP21}.  The first Chern class formula is \cite[Lemma~3.2 (i)]{LR24}.
\end{proof}

We shall use $\mu_H$-semistability only through the following two elementary facts.

\begin{proposition}\label{prop:quot}
Let $X$ be a smooth projective variety, let $H$ be an ample divisor on $X$, and let $E$ be a torsion-free sheaf on $X$. If $E$ is $\mu_H$-semistable and $E\twoheadrightarrow Q$ is a nonzero torsion-free quotient sheaf on $X$, then
$$
  \mu_H(E)\leq \mu_H(Q).
$$
\end{proposition}

\begin{proof}
This is the quotient criterion for slope semistability; see \cite[Proposition~5.7.6(a)]{Kob87}. For completeness, we recall the short argument. Let $K$ be the kernel of $E\twoheadrightarrow Q$. If $K=0$, then there is nothing to prove. Otherwise $K$ is a nonzero subsheaf of $E$, and the $\mu_H$-semistability of $E$ gives $\mu_H(K)\leq\mu_H(E)$. Since
$$
\operatorname{rk}(E)\mu_H(E)=\operatorname{rk}(K)\mu_H(K)+\operatorname{rk}(Q)\mu_H(Q),
$$
we must have $\mu_H(E)\leq\mu_H(Q)$.
\end{proof}

\begin{lemma}[{\cite[Lemma~3.2.2]{HL10}}]\label{lem:pull}
Let $f\colon Y\longrightarrow X$ be a finite morphism of normal projective varieties, and let $H$ be an ample divisor on $X$. Let $\mathcal F$ be a coherent $\mathcal O_X$-module of dimension $\dim X$. Then $\mathcal F$ is $\mu_H$-semistable if and only if $f^\ast\mathcal F$ is $\mu_{f^\ast H}$-semistable.
\end{lemma}

Next we recall two standard facts about exterior powers. The first gives the Chern class calculation used throughout the paper, and the second is used to analyze the relative tangent sequence after passing to an Abelian quotient.

\begin{lemma}[{\cite[Exercise~II.5.16 and Appendix A, C5]{Har77}}]\label{rkc1}
Let $E$ be a vector bundle of rank $n$. Then $\operatorname{rk}\left(\bigwedge^rE\right)=\binom nr$ and $c_1\left(\bigwedge^rE\right)=\binom{n-1}{r-1}c_1(E)$.
\end{lemma}

\begin{lemma}[{\cite[Exercise~II.5.16]{Har77}}]\label{lem:exterior-filtration}
Let
$$
0\longrightarrow A\longrightarrow B\longrightarrow C\longrightarrow 0
$$
be a short exact sequence of vector bundles. Then $\bigwedge^rB$ has a natural filtration
$$
0=F_{r+1}\subset F_r\subset\cdots\subset F_0=\bigwedge^rB
$$
such that
$$
F_i/F_{i+1}\simeq \bigwedge^iA\otimes \bigwedge^{r-i}C
$$
for every $i$.
\end{lemma}

The main structural input is the following theorem of the second author.

\begin{theorem}[{\cite[Theorem~1.3]{Wat23}}]\label{thm:wat23}
Let $X$ be a smooth projective variety of dimension $n$. Assume that $\bigwedge^r T_X$ is nef for some $1\leq r<n$. Then, after taking a suitable finite \'{e}tale cover $\widetilde X \longrightarrow X$, there exists a locally trivial fibration $\varphi\colon \widetilde X \longrightarrow A$ over an Abelian variety $A$ whose fibers are Fano varieties.
\end{theorem}

We shall also use the following elementary consequence of Hodge decomposition. It is included to make explicit the source of the cohomological obstructions in Sections~4 and~5.

\begin{lemma}\label{lem:hodge-nonvanishing}
Let $X$ be a smooth projective variety of dimension $n$, and let $H$ be an ample divisor on $X$. Put $h:=c_1(\mathcal O_X(H))$. For every $0\leq p\leq n$, the class $h^p$ gives a nonzero element of $H^p(X,\Omega_X^p)$. In particular, $H^p(X,\Omega_X^p)\neq 0$.
\end{lemma}

\begin{proof}
The case $p=0$ is clear. For $p>0$, the class $h$ is of type $(1,1)$; hence, by Dolbeault's theorem and the Hodge decomposition for compact K\"ahler manifolds \cite[Corollary~2.6.21 and Section~3.2]{Huy05}, the class $h^p$ may be regarded as an element of $H^p(X,\Omega_X^p)$. If $h^p=0$, then $h^n=h^{n-p}h^p=0$, contradicting
$$
\int_X h^n=H^n>0.
$$
Thus $h^p$ is nonzero.
\end{proof}

Finally we record the ambient version needed for restricted cotangent bundles.

\begin{lemma}\label{lem:ambient-hodge-nonvanishing}
Let $X$ be a smooth projective variety of dimension $n$, let $H$ be a very ample divisor on $X$, and let $\iota\colon X\hookrightarrow \mathbb P^N$ be the embedding defined by $H$. Then, for every $1\leq p\leq n$, one has
$$
H^p\left(X,\Omega_{\mathbb P^N}^p|_X\right)\neq 0.
$$
\end{lemma}

\begin{proof}
Let $h_{\mathbb P^N}:=c_1(\mathcal O_{\mathbb P^N}(1))$ and $h:=\iota^\ast h_{\mathbb P^N}$. The class $h_{\mathbb P^N}^p$ defines an element of $H^p(\mathbb P^N,\Omega_{\mathbb P^N}^p)$, and its pullback gives a class
$$
\alpha_p\in H^p\left(X,\Omega_{\mathbb P^N}^p|_X\right).
$$
Under the natural map $\Omega_{\mathbb P^N}^p|_X\longrightarrow \Omega_X^p$, the class $\alpha_p$ maps to $h^p\in H^p(X,\Omega_X^p)$. This class is nonzero by Lemma~\ref{lem:hodge-nonvanishing}; hence $\alpha_p\neq0$.
\end{proof}

\section{Numerical and cohomological restrictions}
In this section we extract the two main consequences of the Ulrich condition for exterior powers of the tangent bundle. We first keep a line bundle twist in the notation, because the same calculation will be used later to exclude several twisted bundles. Let $M$ be a line bundle on $X$ and, for $0<r<n$, set
$$
E_{r,M}:=\bigwedge^rT_X\otimes M.
$$
When $M=\mathcal O_X$, we write $E_r:=\bigwedge^rT_X$.

\begin{proposition}\label{prop1}
Let $X$ be a smooth projective variety of dimension $n\geq2$, let $H$ be a very ample divisor on $X$, let $M$ be a line bundle on $X$, and let $0<r<n$. Set $E_{r,M}:=\bigwedge^rT_X\otimes M$. If $E_{r,M}$ is $H$-Ulrich, then
$$
 (-K_X)\cdot H^{n-1}=\frac{n}{n+2r}\bigl((n+1)H-2c_1(M)\bigr)\cdot H^{n-1}.
$$
Moreover, $\mu_H(E_{r,M})=\frac{r(n+1)H^n+nc_1(M)\cdot H^{n-1}}{n+2r}$. In particular, for $M=\cO_X$ and $0<r<n$, $(-K_X)\cdot H^{n-1}=\frac{n(n+1)}{n+2r}H^n$, $\mu_H(\bigwedge^r T_X)=\frac{r(n+1)}{n+2r}H^n$.
\end{proposition}

\begin{proof}
By Lemma~\ref{rkc1}, we have
$$
 \rk\left(\bigwedge^r T_X\right)=\binom{n}{r},
 \qquad
 c_1\left(\bigwedge^r T_X\right)=\binom{n-1}{r-1}(-K_X).
$$
Hence $\rk(E_{r,M})=\binom{n}{r}$ and $c_1(E_{r,M})=\binom{n-1}{r-1}(-K_X)+\binom{n}{r}c_1(M)$. Since $E_{r,M}$ is $H$-Ulrich, we obtain
$$
c_1(E_{r,M})\cdot H^{n-1}=\frac{\rk(E_{r,M})}{2}\left(K_X+(n+1)H\right)\cdot H^{n-1}
$$
by Proposition~\ref{prop:c1}. Then 
$$
\left\{\binom{n-1}{r-1}(-K_X)+\binom{n}{r}c_1(M)\right\}\cdot H^{n-1}=\frac{1}{2}\binom{n}{r}\left(K_X+(n+1)H\right)\cdot H^{n-1}.
$$
Dividing by $\binom{n}{r}$ and using $\frac{\binom{n-1}{r-1}}{\binom{n}{r}}=\frac{r}{n}$ we obtain
$$
\left(\frac{r}{n}+\frac{1}{2}\right)(-K_X)\cdot H^{n-1}=\frac{n+1}{2}H^n-c_1(M)\cdot H^{n-1}.
$$
Therefore
$$
 (-K_X)\cdot H^{n-1}=\frac{n}{n+2r}\left((n+1)H-2c_1(M)\right)\cdot H^{n-1}.
$$

It remains to compute the slope.  From the formula for $c_1(E_{r,M})$, we have
$$
\mu_H(E_{r,M})=\frac{c_1(E_{r,M})\cdot H^{n-1}}{\rk(E_{r,M})}=\frac{r}{n}(-K_X)\cdot H^{n-1}+c_1(M)\cdot H^{n-1}.
$$
Using the formula for $(-K_X)\cdot H^{n-1}$ obtained above, we obtain
$$
\begin{aligned}
\mu_H(E_{r,M})&=\frac{r}{n}\cdot\frac{n}{n+2r}\left((n+1)H-2c_1(M)\right)\cdot H^{n-1}+c_1(M)\cdot H^{n-1} \\
&=\frac{r(n+1)H^n-2r\,c_1(M)\cdot H^{n-1}}{n+2r}+c_1(M)\cdot H^{n-1} \\
&=\frac{r(n+1)H^n+nc_1(M)\cdot H^{n-1}}{n+2r}.
\end{aligned}
$$
In particular, if $M=\cO_X$, then $c_1(M)=0$.  Hence $(-K_X)\cdot H^{n-1}=\frac{n(n+1)}{n+2r}H^n$ and $\mu_H\left(\bigwedge^r T_X\right)=\frac{r(n+1)}{n+2r}H^n$.
\end{proof}

We now specialize to the untwisted bundle. The numerical identity gives positive anticanonical degree, while the nefness and $\mu_H$-semistability of the Ulrich bundle eliminate the possible Abelian factor in a structure theorem of varieties with nef exterior powers of the tangent bundles.

\begin{theorem}\label{thm2}
Let $X$ be a smooth projective variety of dimension $n\geq2$, and let $H$ be a very ample divisor on $X$. If $\bigwedge^r T_X$ is $H$-Ulrich for some $0<r<n$, then $X$ is Fano and $(-K_X)\cdot H^{n-1}=\frac{n(n+1)}{n+2r}H^n$.
\end{theorem}

\begin{proof}
By Proposition~\ref{prop1}, $(-K_X)\cdot H^{n-1}=\frac{n(n+1)}{n+2r}H^n$. It remains to prove that $X$ is Fano. Set $E_r:=\bigwedge^rT_X$. Since $E_r$ is $H$-Ulrich, it is globally generated and $\mu_H$-semistable by Proposition~\ref{prop:c1}. In particular, $E_r$ is nef.  

By Theorem~\ref{thm:wat23}, there exists a finite \'{e}tale cover
$f\colon X'\to X$ and a locally trivial fibration $\alpha\colon X'\longrightarrow A$ onto an Abelian variety $A$ whose fibers are Fano varieties.  Set $q:=\dim A$.

Since $f$ is finite and \'{e}tale, we have $f^\ast E_r\simeq \bigwedge^r T_{X'}$ and $K_{X'}=f^\ast K_X$.  Moreover, finite pullback preserves semistability by Lemma~\ref{lem:pull}; hence $f^\ast E_r$ is
$\mu_{f^\ast H}$-semistable.  We have
$$
\begin{aligned}
\mu_{f^\ast H}(f^\ast E_r)&=\frac{c_1(f^\ast E_r)\cdot (f^\ast H)^{n-1}}{\rk(E_r)}  \\
&=(\deg f)\frac{c_1(E_r)\cdot H^{n-1}}{\rk(E_r)}=(\deg f)\mu_H(E_r)>0.
\end{aligned}
$$

Assume that $q>0$.  Since $A$ is an Abelian variety, we have $T_A\simeq \cO_A^{\oplus q}$.  Thus the relative tangent sequence of $\alpha$ is
$$
0\longrightarrow T_{X'/A}\longrightarrow T_{X'}\longrightarrow \alpha^\ast T_A\longrightarrow 0, \qquad\alpha^\ast T_A\simeq \cO_{X'}^{\oplus q}.
$$
By Lemma~\ref{lem:exterior-filtration}, this sequence induces a filtration
$$
0=F_{r+1}\subset F_r\subset \cdots \subset F_{0}=\bigwedge^r T_{X'}
$$
such that
$$
F_i/F_{i+1}\simeq\bigwedge^{i}T_{X'/A}\otimes\bigwedge^{r-i} \cO_{X'}^{\oplus q}.
$$

If $q\geq r$, then $F_0/F_{1}\simeq\bigwedge^r\cO_{X'}^{\oplus q}\simeq \cO_{X'}^{\oplus \binom{q}{r}}$ is a nonzero quotient of $\bigwedge^r T_{X'}$ of slope $0$.  This contradicts the $\mu_{f^\ast H}$-semistability of $\bigwedge^rT_{X'}$. Indeed 
$$
0<\mu_{f^\ast H}\left(\bigwedge^rT_{X'}\right)\leq\mu_{f^\ast H}\left(\bigwedge^r\cO_{X'}^{\oplus q}\right)=0.
$$

Hence we may assume that $0<q<r$. For every $i<r-q$, we have $r-i>q$, and therefore $\bigwedge^{r-i}\mathcal O_{X'}^{\oplus q}=0$. It follows that $F_0=F_1=\cdots=F_{r-q}$. Consequently, $F_{r-q}/F_{r-q+1}\simeq\bigwedge^{r-q}T_{X'/A}$ is a nonzero quotient of $F_0=\bigwedge^rT_{X'}$. Therefore, by semistability,
$$
\mu_{f^\ast H}\left(\bigwedge^rT_{X'}\right)\leq\mu_{f^\ast H}\left(\bigwedge^{r-q}T_{X'/A}\right).
$$
On the other hand, since $K_A\simeq \cO_A$, we have $K_{X'/A}=K_{X'}$.  Hence
$$
\begin{aligned}
\mu_{f^\ast H}\left(\bigwedge^{r-q}T_{X'/A}\right)&=\frac{r-q}{n-q}(-K_{X'/A})\cdot (f^\ast H)^{n-1} \\
&=\frac{r-q}{n-q}(-K_{X'})\cdot (f^\ast H)^{n-1}\\
&<\frac{r}{n}(-K_{X'})\cdot (f^\ast H)^{n-1}.
\end{aligned}
$$
Thus
$$
\mu_{f^\ast H}\left(\bigwedge^{r-q}T_{X'/A}\right)<\frac{r}{n}(-K_{X'})\cdot (f^\ast H)^{n-1}=\mu_{f^\ast H}\left(\bigwedge^rT_{X'}\right),
$$
which is a contradiction.  Therefore $q=0$.

Then the Abelian variety $A$ is a point.  Hence $X'$ is a Fano variety by Theorem~\ref{thm:wat23}.   Thus $-K_{X'}=f^\ast(-K_X)$ is ample.  Since ampleness descends under finite surjective morphisms, $-K_X$ is ample.  Therefore $X$ is Fano.
\end{proof}

The next lemma is only a reformulation of the Ulrich vanishing using the standard identification between exterior powers of the tangent bundle and differential forms.

\begin{lemma}\label{thm3}
Let $X$ be a smooth projective variety of dimension $n\geq2$, let $H$ be a very ample divisor on $X$, let $M$ be a line bundle on $X$, and let $0<r<n$. Set $E_{r,M}:=\bigwedge^rT_X\otimes M$. If $E_{r,M}$ is $H$-Ulrich, then, for any $1\leq i\leq n$ and any $q$,
$$
H^q\left(X, \Omega_X^{n-r}\otimes \mathcal O_X(-K_X)\otimes M(-iH)\right)=0.
$$
\end{lemma}

\begin{proof}
For a vector bundle $E$ of rank $n$, there is a natural isomorphism
$$
\bigwedge^r E\simeq\bigwedge^{n-r}E^{\vee}\otimes \det E.
$$
Applying this to $E=T_X$, we obtain
$$
\bigwedge^r T_X\simeq\bigwedge^{n-r}\Omega_X^1\otimes \det T_X=\Omega_X^{n-r}\otimes \mathcal O_X(-K_X).
$$
Hence, for any $1\leq i\leq n$, we have
$$
\begin{aligned}
E_{r,M}(-iH)&=\left(\bigwedge^r T_X\otimes M\right)(-iH) \\
&\simeq\Omega_X^{n-r}\otimes \mathcal O_X(-K_X)\otimes M(-iH).
\end{aligned}
$$
Since $E_{r,M}$ is $H$-Ulrich, by definition we have
$$
 H^q\left(X,E_{r,M}(-iH)\right)=0
$$
for any $1\leq i\leq n$ and any $q$.  Therefore
$$
H^q\left(X,\Omega_X^{n-r}\otimes \mathcal O_X(-K_X)\otimes M(-iH)\right)=0
$$
for any $1\leq i\leq n$ and any $q$.
\end{proof}

\begin{corollary}\label{cor4}
Let $X$ be a smooth projective variety of dimension $n\geq2$, let $H$ be a very ample divisor on $X$, let $M$ be a line bundle on $X$, and let $0<r<n$. If $M\simeq \mathcal O_X(K_X+jH)$ for some $1\leq j\leq n$, then $\bigwedge^r T_X\otimes M$ is not $H$-Ulrich.
\end{corollary}

\begin{proof}
Put $p:=n-r$.  Suppose that $E_{r,M}:=\bigwedge^rT_X \otimes M$ is
$H$-Ulrich.  By Lemma~\ref{thm3}, applied with $i=j$, we have
$$
H^q\left(X,\Omega_X^p\otimes \cO_X(-K_X)\otimes M(-jH)\right)=0
$$
for every $q$.  Since $M\simeq \cO_X(K_X+jH)$, we obtain $\Omega_X^p\otimes \cO_X(-K_X)\otimes M(-jH)\simeq\Omega_X^p$. Hence $H^q(X,\Omega_X^p)=0$ for every $q$.

By Lemma~\ref{lem:hodge-nonvanishing}, we have $H^p(X,\Omega_X^p)\neq 0$. This contradicts the vanishing obtained above. Hence $E_{r,M}$ is not $H$-Ulrich.
\end{proof}

\section{Twisted exterior powers in the Picard number one case}
We turn to line twists under the additional assumption that the Picard group is generated by the polarization. These results are used both in the proof of the main theorem and to rule out some natural bundles attached to the embedding defined by $H$.

\begin{corollary}\label{cor5}
Let $X$ be a smooth projective variety of dimension $n\geq2$, let $H$ be a very ample divisor on $X$, let $0<r<n$, and let $k\in\Z$. If $\bigwedge^rT_X(kH)$ is $H$-Ulrich, then
$$
(-K_X)\cdot H^{n-1}=\frac{n(n+1-2k)}{n+2r}H^n.
$$
If, in addition, $-K_X=iH$ for some $i\in\Z$, then $i=\frac{n(n+1-2k)}{n+2r}$, and the case $1\leq i+k\leq n$ is impossible.  Consequently, under the assumption $-K_X=iH$,
$$
\bigwedge^rT_X(H)\quad(n\geq2),
\qquad
\bigwedge^rT_X(2H)\quad(n\geq3)
$$
are not $H$-Ulrich.
\end{corollary}

\begin{proof}
Apply Proposition~\ref{prop1} to $M=\mathcal O_X(kH)$. Since $c_1(M)=kH$, we obtain
$$
\begin{aligned}
(-K_X)\cdot H^{n-1}&=\frac{n(n+1-2k)}{n+2r}H^n.
\end{aligned}
$$
This proves the first assertion.

Assume moreover that $-K_X=iH$ for some $i\in \mathbb Z$.  Then $iH^n=(-K_X)\cdot H^{n-1}=\frac{n(n+1-2k)}{n+2r}H^n$. Since $H$ is ample, $H^n>0$.  Hence
$$
i=\frac{n(n+1-2k)}{n+2r}.
$$

We next show that the case $1\leq i+k\leq n$ is impossible.  Suppose that $1\leq i+k\leq n$. Since $-K_X=iH$, we have $\mathcal O_X(kH)\simeq\mathcal O_X(K_X+(i+k)H)$. By Corollary~\ref{cor4}, this contradicts the assumption that $\bigwedge^r T_X(kH)$ is $H$-Ulrich. 

It remains to prove the two consequences.  First, take $k=1$.  If $\bigwedge^r T_X(H)$ were $H$-Ulrich, then the above formula would give $i=\frac{n(n-1)}{n+2r}$. If $n=2$, then $r=1$, and hence $i=1/2$, contradicting
$i\in\mathbb Z$.  If $n\geq 3$, then $0<i=\frac{n(n-1)}{n+2r}<n-1$. Since $i\in\mathbb Z$, this gives $1\leq i\leq n-2$. Thus $1\leq i+1\leq n$, which is impossible by the assertion proved above.  Hence $\bigwedge^r T_X(H)$ is not $H$-Ulrich for $n\geq 2$.

Next, take $k=2$ and assume $n\geq 3$.  If $\bigwedge^r T_X(2H)$ were $H$-Ulrich, then $i=\frac{n(n-3)}{n+2r}$. Since $n\geq 3$, we have $i\geq 0$.  Moreover, $\frac{n(n-3)}{n+2r}<n-2$.  Indeed, this is equivalent to $n(n-3)<(n-2)(n+2r)$, since
$$
(n-2)(n+2r)-n(n-3)=n+2r(n-2)>0.
$$
Thus, since $i\in\mathbb Z$, we have $0\leq i\leq n-2$. Hence $1\leq i+2\leq n$, which is again impossible by the assertion proved above.  Therefore $\bigwedge^r T_X(2H)$ is not $H$-Ulrich for $n\geq 3$.
\end{proof}

\begin{proposition}\label{prop6}
Let $X$ be a smooth projective variety of dimension $n\geq2$, and let $H$ be a very ample divisor on $X$. Assume that $\operatorname{Pic}(X)=\mathbb Z[H]$, let $0<r<n$, and let $k\in\mathbb Z$. If $\bigwedge^r T_X(kH)$ is $H$-Ulrich, then $k\geq n+2$.
\end{proposition}

\begin{proof}
Set $E:=\bigwedge^r T_X(kH)$. Since $E$ is $H$-Ulrich, it is globally generated.  Hence $\det E$ is globally generated, and in particular $c_1(E)\cdot H^{n-1}\geq 0$. Equivalently, $\mu_H(E)\geq 0$. By Proposition~\ref{prop1}, applied to $M=\mathcal O_X(kH)$, we have
$$
\mu_H(E)=\frac{r(n+1)H^n+n k H^n}{n+2r}=\frac{n(r+k)+r}{n+2r}H^n.
$$
Since $H^n>0$, it follows that $n(r+k)+r\geq 0$. As $k\in \mathbb Z$ and $0<r<n$, this implies $k\geq -r$. Indeed, if $k\leq -r-1$, then $n(r+k)+r\leq -n+r<0$, which is a contradiction.

Since $\Pic(X)=\mathbb Z[H]$, we may write $-K_X=iH$ for some $i\in\mathbb Z$.  By Corollary~\ref{cor5}, we have
$$
i=\frac{n(n+1-2k)}{n+2r}.
$$

Put $j_{r,k}:=i+k$. Then
$$
\begin{aligned}
j_{r,k}&=\frac{n(n+1-2k)}{n+2r}+k \\
&=\frac{n(n+1)+(2r-n)k}{n+2r}.
\end{aligned}
$$

We claim that if $-r\leq k\leq n+1$, then $1\leq j_{r,k}\leq n$. Since $j_{r,k}$ is affine in $k$, it is enough to check the two endpoints $k=n+1$ and $k=-r$.

First, for $k=n+1$, we have $j_{r,n+1}=\frac{2r(n+1)}{n+2r}$. Then $j_{r,n+1}\geq 1$ is equivalent to $2r(n+1)\geq n+2r$, that is, $n(2r-1)\geq 0$, which is clear.  Also, $j_{r,n+1}\leq n$ is equivalent to $2r(n+1)\leq n(n+2r)$, that is, $n^2-2r\geq 0$. Since $0<r<n$ and $n\geq 2$, we indeed have $n^2-2r\geq n^2-2(n-1)>0$. Thus
$$
1\leq j_{r,n+1}\leq n.
$$

Next, for $k=-r$, we have $j_{r,-r}=\frac{n(n+1)+r(n-2r)}{n+2r}$. We prove first that $j_{r,-r}\geq 1$.  This is equivalent to $n(n+1)+r(n-2r)\geq n+2r$, or $n^2+nr-2r^2-2r\geq 0$. Set
$$
g(r):=n^2+nr-2r^2-2r.
$$
As a polynomial in $r$, the function $g(r)$ is concave. Hence, on the interval $1\leq r\leq n-1$, it is enough to check the endpoints.  We have $g(1)=n^2+n-4>0$ and $g(n-1)=n>0$. Therefore $j_{r,-r}\geq 1$. Moreover, $ j_{r,-r}\leq n$ is equivalent to $n(n+1)+r(n-2r)\leq n(n+2r)$, or $n(r-1)+2r^2\geq 0$, which is clear.  Thus
$$
1\leq j_{r,-r}\leq n.
$$
This proves the claim.

Now suppose that $k\leq n+1$. Since we have already proved $k\geq -r$, the claim gives
$$
1\leq j_{r,k}=i+k\leq n.
$$
This contradicts Corollary~\ref{cor5}.  Hence $k\geq n+2$.
\end{proof}


\begin{corollary}\label{cor7}
Let $X$ be a smooth projective variety of dimension $n$ with $2\leq n\leq 12$, and let $H$ be a very ample divisor on $X$. Assume that $\operatorname{Pic}(X)=\mathbb Z[H]$. Then $T_X(kH)$ is not $H$-Ulrich for any $k\in\mathbb Z$.
\end{corollary}

\begin{proof}
Suppose that $T_X(kH)$ is $H$-Ulrich for some $k\in \mathbb Z$.  Since $T_X(kH)=\bigwedge^1 T_X(kH)$, Proposition~\ref{prop6}, applied with $r=1$, gives $k\geq n+2$. 

On the other hand, the very ample divisor $H$ gives an embedding $X\hookrightarrow \mathbb P^N$ such that $\mathcal O_X(H)\simeq \mathcal O_{\mathbb P^N}(1)|_X$. Since $2\leq n\leq 12$, by \cite[Theorem~2]{LR24}, if $T_X(kH)$ is $H$-Ulrich, then $k\leq n+1$. This contradicts $k\geq n+2$.  Therefore $T_X(kH)$ is not $H$-Ulrich for any $k\in \mathbb Z$.
\end{proof}

\section{The Picard number one case}
We finish the proof of the classification in Picard number one. Throughout this section, $X$ is a smooth projective variety of dimension $n\geq2$, $H$ is very ample, and $\bigwedge^rT_X$ is assumed to be $H$-Ulrich for some $0<r<n$. By Theorem~\ref{thm2}, the variety $X$ is Fano.

\begin{proposition}\label{prop12}
Let $X$ be a smooth projective variety of dimension $n\geq2$ with $\rho(X)=1$, and let $H$ be a very ample divisor on $X$. Assume that $\bigwedge^r T_X$ is $H$-Ulrich for some $0<r<n$. Then $\operatorname{Pic}(X)$ is generated by an ample line bundle $L$. Writing $\operatorname{Pic}(X)=\mathbb Z[L]$, $-K_X=i_XL$, and $H=aL$, we have $a=1$ or $a=2$. If $a=1$, then
$$
i_X=\frac{n(n+1)}{n+2r}.
$$
If $a=2$, then
$$
i_X=\frac{2n(n+1)}{n+2r}
\quad\text{and}\quad n\leq 2r.
$$
Furthermore, if $a=2$ and $n=2r$, then $(X,H)\simeq (\mathbb P^n,\mathcal O_{\mathbb P^n}(2))$. 
\end{proposition}

\begin{proof}
By Theorem~\ref{thm2}, $X$ is Fano and $(-K_X)\cdot H^{n-1}=\frac{n(n+1)}{n+2r}H^n$. Since $X$ is Fano, its Picard group is torsion-free. Thus $\rho(X)=1$ implies that we may write $\operatorname{Pic}(X)=\mathbb Z[L]$, $-K_X=i_XL$, $H=aL$ where $L$ is the ample generator, $i_X>0$, and $a>0$.

Substituting $-K_X=i_XL$ and $H=aL$ into the above equality, we obtain $i_XL\cdot (aL)^{n-1}=\frac{n(n+1)}{n+2r}(aL)^n$. Hence
$$
i_X a^{n-1}L^n=\frac{n(n+1)}{n+2r}a^nL^n.
$$
Since $a>0$ and $L^n>0$, this gives
$$
  i_X=
  \frac{n(n+1)a}{n+2r}.
$$

By the Kobayashi--Ochiai theorem \cite{KO73}, the Fano index satisfies $i_X\leq n+1$. Therefore $\frac{n(n+1)a}{n+2r}\leq n+1$. Equivalently, $na\leq n+2r$ or $n(a-1)\leq 2r$. Since $0<r<n$, we have $2r<2n$.  Hence $a-1<2$. As $a$ is a positive integer, it follows that 
$$
a=1 \quad\text{or}\quad a=2.
$$

If $a=1$, then the formula for $i_X$ gives $i_X=\frac{n(n+1)}{n+2r}$. 

If $a=2$, then $i_X=\frac{2n(n+1)}{n+2r}$. Moreover, applying again the inequality $i_X\leq n+1$, we obtain
$$
\frac{2n(n+1)}{n+2r}\leq n+1.
$$
Thus $2n\leq n+2r$, and hence $n\leq 2r$. 

Finally, assume that $a=2$ and $n=2r$.  Then
$$  
i_X= \frac{2n(n+1)}{n+2r}=\frac{2n(n+1)}{2n}=n+1.
$$
By the equality case of the Kobayashi--Ochiai theorem \cite{KO73}, we have $X\simeq \mathbb P^n$ and the ample generator $L$ corresponds to $\mathcal O_{\mathbb P^n}(1)$. Since $H=2L$, it follows that $(X,H)\simeq\left(\mathbb P^n,\mathcal O_{\mathbb P^n}(2)\right)$. 
\end{proof}

\begin{lemma}\label{lem13}
Let $r\geq 2$ be an integer. Then the bundle $\bigwedge^r T_{\mathbb P^{2r}}$ is not
$\mathcal O_{\mathbb P^{2r}}(2)$-Ulrich.
\end{lemma}

\begin{proof}
Put $P:=\mathbb P^{2r}$. Suppose that $\bigwedge^r T_P$ is $\mathcal O_P(2)$-Ulrich. Consider the Euler sequence on $P$:
$$
0\longrightarrow\mathcal O_P \longrightarrow\mathcal O_P(1)^{\oplus 2r+1}\longrightarrow T_P\longrightarrow 0.
$$
Taking the $r$-th exterior power of the surjection
$$
\mathcal O_P(1)^{\oplus 2r+1}\twoheadrightarrow T_P,
$$
we obtain a surjection
$$
\bigwedge^r\left(\mathcal O_P(1)^{\oplus 2r+1}\right)\twoheadrightarrow \bigwedge^r T_P.
$$
Since $\bigwedge^r\left(\mathcal O_P(1)^{\oplus 2r+1}\right)\simeq\mathcal O_P(r)^{\oplus \binom{2r+1}{r}}$, we obtain $\mathcal O_P(r)^{\oplus \binom{2r+1}{r}}\twoheadrightarrow\bigwedge^r T_P$. Twisting by $\mathcal O_P(-2)$, we obtain
$$
\mathcal O_P(r-2)^{\oplus \binom{2r+1}{r}}\twoheadrightarrow\bigwedge^r T_P(-2).
$$
Since $r\geq 2$, the line bundle $\mathcal O_P(r-2)$ is globally generated.  Hence $\bigwedge^r T_P(-2)$ is globally generated as a quotient of a globally generated vector bundle.  Therefore
$$
H^0\left(P,\bigwedge^r T_P(-2)\right)\neq 0.
$$

On the other hand, if $\bigwedge^r T_P$ is $\mathcal O_P(2)$-Ulrich, then the Ulrich vanishing with $j=1$ gives $H^q\left(P,\bigwedge^r T_P\otimes \mathcal O_P(-2)\right)=0$ for every $q$.  In particular,
$$
H^0\left(P,\bigwedge^r T_P(-2)\right)=0,
$$
which contradicts the non-vanishing above.  Hence $\bigwedge^r T_{\mathbb P^{2r}}$ is not $\mathcal O_{\mathbb P^{2r}}(2)$-Ulrich.
\end{proof}

\begin{proposition}\label{prop14}
Let $X$ be a smooth projective variety of dimension $n\geq3$, and let $H$ be a very ample divisor on $X$. Assume that $2\leq r<n$, $\operatorname{Pic}(X)=\mathbb Z[L]$,
$H=2L$, and $-K_X=i_XL$. Then $\bigwedge^r T_X$ is not $H$-Ulrich.
\end{proposition}

\begin{proof}
Suppose that $E:=\bigwedge^r T_X$ is $H$-Ulrich.  By Proposition~\ref{prop12}, applied with $a=2$, we have $i_X=\frac{2n(n+1)}{n+2r}$ and $n\leq 2r$. 

\begin{claim}\label{c1}
The index $i_X$ is odd.
\end{claim}

\begin{proof}
Suppose that $i_X$ is even.  Write $i_X=2m$ for some $m\in\mathbb Z$.  Since $-K_X=i_XL$ and $H=2L$, we obtain $-K_X=mH$. Equivalently, $\cO_X\simeq \cO_X(K_X+mH)$. Moreover, by \cite{KO73}, $i_X\leq n+1$.  Hence
$$
1\leq m=\frac{i_X}{2}\leq \frac{n+1}{2}\leq n.
$$
Thus Corollary~\ref{cor4}, applied to $M=\cO_X$ and $j=m$, implies that $\bigwedge^r T_X$ is not $H$-Ulrich.  This contradicts our assumption.  Therefore $i_X$ is odd.
\end{proof}

\begin{claim}\label{c2}
$r<n<2r$.
\end{claim}

\begin{proof}
The inequality $r<n$ is part of the assumption.  By Proposition~\ref{prop12}, we have $n\leq 2r$. Suppose that $n=2r$.  Then Proposition~\ref{prop12} gives $(X,H)\simeq \left(\mathbb P^{2r},\cO_{\mathbb P^{2r}}(2)\right)$. Since $r\geq2$, Lemma~\ref{lem13} says that $\bigwedge^r T_{\mathbb P^{2r}}$ is not $\cO_{\mathbb P^{2r}}(2)$-Ulrich.  This contradicts the assumption that $\bigwedge^rT_X$ is $H$-Ulrich.  Hence $n<2r$.
\end{proof}

\begin{claim}\label{c3}
The integer $n$ is even.
\end{claim}

\begin{proof}
Suppose that $n$ is odd.  Then $n+2r$ is also odd.  Since $i_X=\frac{2n(n+1)}{n+2r}$ is an integer and $\gcd(n+2r,2)=1$, the integer $n+2r$ divides $n(n+1)$.  Hence $i_X=2\cdot \frac{n(n+1)}{n+2r}$ is even.  This contradicts the previous claim.  Therefore $n$ is even.
\end{proof}

Set $C:=n+1-i_X$, $S:=2r-n$. By Claim~\ref{c2}, $S>0$.  Moreover, since $i_X=\frac{2n(n+1)}{n+2r}$ and $n<2r$, we have $i_X<n+1$, so $C>0$.  By Claim~\ref{c1} and Claim~\ref{c3}, both $C$ and $S$ are even integers.

We now rewrite the equality $(n+2r)i_X=2n(n+1)$. Since
$$
n+2r=2n+S
\qquad\text{and}\qquad
i_X=n+1-C,
$$
we obtain $(2n+S)(n+1-C)=2n(n+1)$. Thus $S(n+1-C)=2nC$. Equivalently, $n(S-2C)=S(C-1)$. Since $C$ is a positive even integer, we have $C\geq2$.  Therefore $S(C-1)>0$, and hence $S>2C$. If $S=2C+2$, then $2n=2(C+1)(C-1)$, so $n=C^2-1$. But $C$ is even, and hence $C^2-1$ is odd.  This contradicts the fact that $n$ is even.  Therefore $S\neq 2C+2$. Since $S$ and $C$ are even and $S>2C$, we conclude that $S\geq 2C+4$. 

\begin{claim}\label{c4}
There exists an odd integer $m$ with $1\leq m\leq S-C-3$ such that $H^0(X,mL)\neq 0$.
\end{claim}

\begin{proof}
Suppose that $H^0(X,mL)=0$ for every odd integer $m$ with
$$
1\leq m\leq S-C-3.
$$
Notice that $S-C-3$ is a positive odd integer because $S$ and $C$ are even and $S\geq 2C+4$. Let $P(t):=\chi(X,tL)$ be the Hilbert polynomial with respect to $L$.  Let $m$ be a positive odd
integer with $1\leq m\leq S-C-3$. Since $i_X$ is odd, the integer $\frac{i_X+m}{2}$ is positive.  Moreover,
$$
mL=K_X+\frac{i_X+m}{2}H.
$$
By Kodaira vanishing, we have $H^q(X,mL)=0$ for every $q>0$.  By our assumption, $H^0(X,mL)=0$.  Hence $P(m)=0$. Since $n$ is even, Serre duality gives
$$
P(t)=P(-i_X-t).
$$
Therefore $P(-i_X-m)=0$ for every positive odd integer $m$ with $1\leq m\leq S-C-3$.

Next, let $j$ be an integer with $1\leq j\leq i_X-1$. Then $-jL=K_X+(i_X-j)L$. Since $(i_X-j)L$ is ample, Kodaira vanishing gives $H^q(X,-jL)=0$ for every $q>0$.  On the other hand, $H^0(X,-jL)=0$ because $-jL$ is anti-ample.  Hence $P(-j)=0$ for every $1\leq j\leq i_X-1$.

Thus $P(t)$ has at least $(i_X-1)+2\cdot\frac{S-C-2}{2}=i_X+S-C-3$ distinct roots.  Since $i_X=n+1-C$, this number is
$$
n+S-2C-2.
$$
Using $S\geq 2C+4$, we obtain $n+S-2C-2\geq n+2$. This is impossible because $P(t)$ is a nonzero polynomial of degree $n$. Therefore there exists an odd integer $m$ with $1\leq m\leq S-C-3$ such that $H^0(X,mL)\neq0$. 
\end{proof}

Now put $p:=n-r$. By Claim~\ref{c4}, we can choose an odd integer $m$ such that
$$
1\leq m\leq S-C-3\qquad\text{and}\qquad H^0(X,mL)\neq0.
$$
Set $q:=\frac{i_X-2-m}{2}$. Since $i_X$ and $m$ are odd, $q$ is an integer.  Moreover, 
$$
\begin{aligned}
q-p&=\frac{i_X-2-m}{2}-(n-r) \\
&=\frac{n+1-C-2-m-2n+2r}{2} \\
&=\frac{S-C-1-m}{2}.
\end{aligned}
$$
Since $m\leq S-C-3$, we have $q-p\geq1$. Hence $q>p$. 

Let $X\hookrightarrow \mathbb P^N$ be the embedding defined by $H$.  Since $q>p$, we see that $\Omega_{\mathbb P^N}^p(q)$ is globally generated. We recall this well-known fact for the reader's convenience. Indeed, applying Lemma~\ref{lem:exterior-filtration} to the Euler sequence
$$
0\longrightarrow \Omega_{\mathbb P^N}^1(1)
\longrightarrow
H^0(\mathbb P^N,\mathcal O_{\mathbb P^N}(1))\otimes \mathcal O_{\mathbb P^N}
\longrightarrow
\mathcal O_{\mathbb P^N}(1)
\longrightarrow 0,
$$
we obtain an exact sequence
$$
0
\longrightarrow
\Omega_{\mathbb P^N}^{p+1}(p+1)
\longrightarrow
\bigwedge^{p+1}H^0(\mathbb P^N,\mathcal O_{\mathbb P^N}(1))\otimes \mathcal O_{\mathbb P^N}
\longrightarrow
\Omega_{\mathbb P^N}^{p}(p+1)
\longrightarrow
0.
$$
After tensoring it with $\mathcal O_{\mathbb P^N}(q-p-1)$, the middle term becomes globally generated, since $q-p-1\geq 0$. This implies that the quotient $\Omega_{\mathbb P^N}^p(q)$ is globally generated.

Hence its restriction to $X$ is globally generated. The natural surjection
$$
\Omega_{\mathbb P^N}^p|_X\longrightarrow\Omega_X^p
$$
then shows that $\Omega_X^p(qH)$ is globally generated. Since $p=n-r$, this is a nonzero vector bundle.

Together with $H^0(X,mL)\neq0$, we obtain $H^0\left(X,\Omega_X^p(qH+mL)\right)\neq0$. By the definition of $q$, we have
$$
qH+mL=(2q+m)L=(i_X-2)L.
$$
Thus
$$
H^0\left(X,\Omega_X^p((i_X-2)L)\right)\neq0.
$$

On the other hand, since $p=n-r$, we have
$$
\bigwedge^r T_X\simeq\Omega_X^p\otimes \cO_X(-K_X)=\Omega_X^p(i_XL).
$$
Therefore
$$
E(-H)=\bigwedge^r T_X(-2L)\simeq\Omega_X^p((i_X-2)L).
$$
Hence $H^0(X,E(-H))\neq0$. This contradicts the Ulrich vanishing for $E$, because an $H$-Ulrich bundle satisfies $H^0(X,E(-H))=0$. Therefore $\bigwedge^rT_X$ is not $H$-Ulrich.
\end{proof}

\begin{theorem}\label{main}
Let $X$ be a smooth projective variety of dimension $n\geq 2$ with $\rho(X)=1$, and let $H$ be a very ample divisor on $X$. Assume that $\bigwedge^r T_X$ is $H$-Ulrich for some $0<r<n$. Then $(X,H) \simeq \left(\mathbb P^2,\mathcal O_{\mathbb P^2}(2)\right)$. Conversely, $T_{\mathbb P^2}$ is $\mathcal O_{\mathbb P^2}(2)$-Ulrich.
\end{theorem}

\begin{proof}
By Theorem~\ref{thm2}, $X$ is Fano. By Proposition~\ref{prop12}, we may write $\operatorname{Pic}(X)=\mathbb Z[L]$, $-K_X=i_XL$, $H=aL$ with $a=1$ or $a=2$.

If $a=1$, then $\Pic(X)=\Z[H]$.  Proposition~\ref{prop6}, applied with $k=0$, shows that this case cannot occur.  Hence $a=2$.

If $r\geq2$, then Proposition~\ref{prop14} gives a contradiction.  Thus $r=1$.  By Proposition~\ref{prop12}, we have $n\leq2r=2$. Since $n\geq2$, this gives $n=2$.  Then $a=2$ and $n=2r$, so Proposition~\ref{prop12} yields $(X,H)\simeq(\P^2,\cO_{\P^2}(2))$. 

Finally, we check the converse.  Let $h$ denote the hyperplane class on $\P^2$.  We must show that
$$
H^q(\P^2,T_{\P^2}(-2h))=0, \qquad H^q(\P^2,T_{\P^2}(-4h))=0
$$
for every $q$.  Twisting the Euler sequence
$$
0\longrightarrow\cO_{\P^2}\longrightarrow\cO_{\P^2}(1)^{\oplus3}\longrightarrow T_{\P^2}\longrightarrow0
$$
by $\cO_{\P^2}(-2)$ gives
$$
0\longrightarrow\cO_{\P^2}(-2)\longrightarrow\cO_{\P^2}(-1)^{\oplus3}\longrightarrow T_{\P^2}(-2)\longrightarrow0.
$$
The cohomology of $\cO_{\P^2}(-2)$ and $\cO_{\P^2}(-1)$ vanishes in every degree. Hence
$$
H^q(\P^2,T_{\P^2}(-2))=0
$$
for every $q$.

Twisting the Euler sequence by $\cO_{\P^2}(-4)$ gives
$$
0\longrightarrow\cO_{\P^2}(-4)\longrightarrow\cO_{\P^2}(-3)^{\oplus3}\longrightarrow T_{\P^2}(-4)\longrightarrow0.
$$
Here the $H^0$ and $H^1$ groups of the two line bundles on the left vanish. For $H^2$, the induced map is dual, by Serre duality, to the natural multiplication map
$$
H^0(\P^2,\cO_{\P^2})^{\oplus3}\longrightarrow H^0(\P^2,\cO_{\P^2}(1)),
$$
which is an isomorphism. Hence
$$
H^q(\P^2,T_{\P^2}(-4))=0
$$
for every $q$.  Therefore $T_{\P^2}$ is $\cO_{\P^2}(2)$-Ulrich.
\end{proof}

\section{Further non-existence results}\label{sec:further}
The results in this section are independent of the Picard number one classification. We collect them here in order to separate these general non-existence statements from the proof of Theorem~\ref{main}.

Let us recall some related results. Benedetti et al. proved that the cotangent bundle is never Ulrich and classified the case where the tangent bundle itself is Ulrich \cite{BMPT23}. Casnati gave a characteristic-free treatment of tangent, cotangent, normal, and conormal bundles from the viewpoint of instanton bundles \cite{Cas24}. Twisted versions have also been studied in several directions: Lopez characterized the varieties for which $N_X(-kH)$ is Ulrich \cite{Lop25Normal}, Antonelli et al. proved that, in the nondegenerate case, the Ulrichness of $N_X^\vee(kH)$ forces $X$ to be a curve \cite{ACLR24}, and Lopez--Raychaudhury studied $T_X(kH)$ \cite{LR24}. More recently, Torres-Lopez--Zamora studied twisted syzygy and dual syzygy bundles \cite{TLZ26}, which for the embedding $X\subset\P^N$ are related to twists of $\Omega_{\P^N}^1(1)|_X$ and $T_{\P^N}(-1)|_X$, while Sarkar studied extensions of Ulrich bundles to the ambient projective space and related applications to restricted ambient bundles \cite{Sar26}. The elementary arguments below are included because they give the precise obstructions needed here for higher exterior powers, tensor powers, and the range of twists used in the preceding sections.

\begin{proposition}\label{prop9}
Let $X$ be a smooth projective variety of dimension $n$, and let $H$ be a very ample divisor on $X$. For any $0\leq p\leq n$ and $1\leq j\leq n$, $\Omega_X^p(jH)$ is not $H$-Ulrich. Moreover, for any $1\leq p\leq n$ and $1\leq j\leq n$, $(\Omega_X^1)^{\otimes p}(jH)$ is not $H$-Ulrich.
\end{proposition}

\begin{proof}
We first prove that $\Omega_X^p(jH)$ is not $H$-Ulrich. If $p=0$, then $\Omega_X^0(jH)\simeq\cO_X(jH)$. If this line bundle were $H$-Ulrich, then the Ulrich vanishing with $i=j$ would give $H^0(X,\cO_X)=0$, a contradiction.

Assume next that $1\leq p<n$.  Put
$$
r:=n-p \qquad\text{and}\qquad M:=\cO_X(K_X+jH).
$$
Then $1\leq r<n$, and
$$
\bigwedge^r T_X\otimes M\simeq\Omega_X^{n-r}\otimes \cO_X(-K_X)\otimes \cO_X(K_X+jH)\simeq\Omega_X^p(jH).
$$
Hence the assertion follows from Corollary~\ref{cor4}. It remains to consider the case $p=n$.  Suppose that $\Omega_X^n(jH)\simeq \cO_X(K_X+jH)$ is $H$-Ulrich.  Taking $i=j$ in the Ulrich vanishing, we obtain
$$
  H^q(X,\cO_X(K_X))=0
$$
for every $q$.  This contradicts Serre duality, since $H^n(X,\cO_X(K_X))\simeq H^0(X,\cO_X)^{\vee}\neq 0$. Thus $\Omega_X^p(jH)$ is not $H$-Ulrich for every $0\leq p\leq n$.

It remains to prove the assertion for tensor powers. Assume that $(\Omega_X^1)^{\otimes p}(jH)$ is $H$-Ulrich for some $1\leq p\leq n$ and $1\leq j\leq n$.  Twisting by $-jH$, the Ulrich vanishing gives
$$
H^q\left(X,(\Omega_X^1)^{\otimes p}\right)=0
$$
for every $q$.  In characteristic zero, $\Omega_X^p$ is a direct summand of $(\Omega_X^1)^{\otimes p}$ via the alternating idempotent. Hence $H^p(X,\Omega_X^p)$ is a direct summand of $H^p\left(X,(\Omega_X^1)^{\otimes p}\right)$.  By Lemma~\ref{lem:hodge-nonvanishing}, $H^p(X,\Omega_X^p)\neq0$, a contradiction. Therefore $(\Omega_X^1)^{\otimes p}(jH)$ is not $H$-Ulrich.
\end{proof}

\begin{proposition}\label{prop10}
Let $X$ be a smooth projective variety of dimension $n$, let $H$ be a very ample divisor on $X$, and let $\iota\colon X\hookrightarrow\mathbb P^N$ be the embedding defined by $H$. For any $0\leq p\leq n$ and $1\leq j\leq n$, $\Omega_{\mathbb P^N}^p|_X(jH)$ is not $H$-Ulrich.
\end{proposition}

\begin{proof}
Let $\iota\colon X\hookrightarrow \mathbb P^N$ be the embedding defined by $H$.  Put
$$
h_{\mathbb P^N}:=c_1(\mathcal O_{\mathbb P^N}(1)),
\qquad
h:=\iota^\ast h_{\mathbb P^N}=c_1(\mathcal O_X(H)).
$$

First, we consider the case $p=0$.  Suppose that $\Omega_{\mathbb P^N}^0|_X(jH)\simeq \mathcal O_X(jH)$ is $H$-Ulrich.  Taking $i=j$ in the Ulrich vanishing, we obtain $H^q(X,\mathcal O_X)=0$ for every $q$.  This contradicts $H^0(X,\mathcal O_X)\neq 0$. Hence $\mathcal O_X(jH)$ is not $H$-Ulrich.

Assume next that $1\leq p\leq n$. By Lemma~\ref{lem:ambient-hodge-nonvanishing}, we have
$$
  H^p\left(X,\Omega_{\mathbb P^N}^p|_X\right)\neq 0.
$$
If $\Omega_{\mathbb P^N}^p|_X(jH)$ were $H$-Ulrich, then twisting by $-jH$ would give $H^p\left(X,\Omega_{\mathbb P^N}^p|_X\right)=0$, a contradiction. Therefore $\Omega_{\mathbb P^N}^p|_X(jH)$ is not $H$-Ulrich.
\end{proof}

\begin{proposition}\label{prop11}
Let $X$ be a smooth projective variety of dimension $n$, let $H$ be a very ample divisor on $X$, and let $\iota\colon X\hookrightarrow\mathbb P^N$ be the embedding defined by $H$. Put $N_X:=N_{X/\mathbb P^N}$. Let $B=T_{\mathbb P^N}|_X$ or $B=N_X$. If $1\leq m\leq\rk B$ and $1\leq k+m\leq n$, then $\bigwedge^mB(kH)$ is not $H$-Ulrich.
\end{proposition}

\begin{proof}
Since $H$ defines the embedding $\iota\colon X\hookrightarrow \mathbb P^N$, we have $\mathcal O_X(H)\simeq \mathcal O_{\mathbb P^N}(1)|_X$. First we show that $B(-H)$ is globally generated.  By the Euler sequence on $\mathbb P^N$,
$$
0\longrightarrow\mathcal O_{\mathbb P^N}(-1)\longrightarrow\mathcal O_{\mathbb P^N}^{\oplus N+1}\longrightarrow T_{\mathbb P^N}(-1)\longrightarrow 0,
$$
the vector bundle $T_{\mathbb P^N}(-1)$ is globally generated.  Hence $T_{\mathbb P^N}(-1)|_X$ is globally generated.  Therefore, if $B=T_{\mathbb P^N}|_X$, then $B(-H)=T_{\mathbb P^N}(-1)|_X$ is globally generated.

If $B=N_X$, then the normal sequence
$$
0\longrightarrow T_X \longrightarrow T_{\mathbb P^N}|_X\longrightarrow N_X\longrightarrow 0
$$
gives, after twisting by $\mathcal O_X(-H)$, a surjection
$$
T_{\mathbb P^N}(-1)|_X\longrightarrow N_X(-H)\longrightarrow 0.
$$
Since $T_{\mathbb P^N}(-1)|_X$ is globally generated, its quotient $N_X(-H)$ is also globally generated.  Thus, in both cases, $B(-H)$ is globally generated.

Now let $1\leq m\leq \rk B$.  Then $\bigwedge^m B(-mH)= \bigwedge^m(B(-H))$ is a nonzero globally generated vector bundle.  In particular,
$$
  H^0\left(X,\bigwedge^m B(-mH)\right)\neq 0.
$$

Suppose that $E:=\bigwedge^m B(kH)$ is $H$-Ulrich.  Put $j:=k+m$. By assumption, $1\leq j\leq n$. Hence, by the definition of an $H$-Ulrich bundle,
$$
  H^q\left(X,E(-jH)\right)=0
$$
for every $q$.  In particular, $H^0\left(X,E(-jH)\right)=0$. On the other hand, since $j=k+m$, we have
$$
\begin{aligned}
E(-jH)&=\bigwedge^m B(kH-jH) \\
&=\bigwedge^m B(-mH).
\end{aligned}
$$
Therefore
$$
H^0\left(X,E(-jH)\right)=H^0\left(X,\bigwedge^m B(-mH)\right)\neq 0,
$$
which is a contradiction.  Hence $\bigwedge^m B(kH)$ is not $H$-Ulrich.
\end{proof}

We end with a short observation on Enriques surfaces. It explains why the natural Ulrich bundle found by Borisov--Nuer is not excluded by Proposition~\ref{prop9}.

\begin{proposition}\label{prop:enriques}
Let $Y$ be an Enriques surface, let $H$ be a very ample divisor on $Y$, and let $L$ be a line bundle on $Y$. If $\Omega_Y^1\otimes L$ is $H$-Ulrich, then
$$
2c_1(L)\cdot H=3H^2.
$$
In particular, for every $k\in\mathbb Z$, neither $\Omega_Y^1(kH)$ nor $T_Y(kH)$ is $H$-Ulrich.
\end{proposition}

\begin{proof}
Since $Y$ is an Enriques surface, the canonical divisor $K_Y$ is numerically trivial. Suppose that $E:=\Omega_Y^1\otimes L$ is $H$-Ulrich. Then $\operatorname{rk}(E)=2$ and
$$
c_1(E)=K_Y+2c_1(L).
$$
By Proposition~\ref{prop:c1}, applied to the surface $Y$, we have
$$
c_1(E)\cdot H=\frac{2}{2}(K_Y+3H)\cdot H=3H^2.
$$
Since $K_Y\cdot H=0$, this gives $2c_1(L)\cdot H=3H^2$.

If $L\simeq\mathcal O_Y(kH)$, then the equality above becomes $2kH^2=3H^2$, hence $2k=3$, which is impossible for $k\in\mathbb Z$. Thus $\Omega_Y^1(kH)$ is not $H$-Ulrich. The same argument applies to $T_Y(kH)$ because $c_1(T_Y(kH))=-K_Y+2kH$ and $K_Y\cdot H=0$.
\end{proof}

\begin{remark}\label{rem:enriques-bn}
Proposition~\ref{prop:enriques} only rules out twists by integral multiples of the chosen polarization $H$. This is compatible with the result of Borisov--Nuer. If $\Delta$ is a Fano polarization on an unnodal Enriques surface $Y$, then Borisov--Nuer prove that
$$
\Omega_Y^1(3\Delta)
$$
is a stable Ulrich bundle with respect to the very ample divisor $H=2\Delta$ \cite[Theorem~4.1]{BN18}. In this case the numerical condition in Proposition~\ref{prop:enriques} is exactly satisfied:
$$
2(3\Delta)\cdot (2\Delta)=3(2\Delta)^2.
$$
Thus the twist $3\Delta$ is not an integral multiple of $H=2\Delta$, and so it is not covered by the non-existence statement for $\Omega_Y^1(kH)$.
\end{remark}

\section*{Declaration of generative AI and AI-assisted technologies in the manuscript preparation process}
In late April and early May 2026, the second author used ChatGPT Pro
(GPT-5.5 Pro) as a research-assistance tool to discuss possible strategies for
the proof. In particular, the authors initially considered case-by-case
arguments based on the intersection-number restrictions for $-K_X$ in
\cite{LR24}. These arguments reached the case $r=2$, where the relevant
coindex is at most three, but they relied on classification results and did not
seem to extend uniformly to $r\geq3$. In subsequent discussions with the tool,
the use of the Hodge decomposition was considered, leading the authors to a
more uniform argument independent of $r$ and to the proof of
Proposition~\ref{prop14}. After using this tool, the authors checked all
mathematical arguments, edited the content as needed, and take full
responsibility for the content of the paper.

\bibliographystyle{amsalpha}
\bibliography{biblio}

@article{Bea18,
  author  = {Beauville, Arnaud},
  title   = {An introduction to {U}lrich bundles},
  journal = {Eur. J. Math.},
  volume  = {4},
  number  = {1},
  pages   = {26--36},
  year    = {2018},
  doi     = {10.1007/s40879-017-0154-4}
}

@article{BMPT23,
  author  = {Benedetti, Vladimiro and Montero, Pedro and {Prieto-Monta\~nez}, Yulieth and Troncoso, Sergio},
  title   = {Projective manifolds whose tangent bundle is {U}lrich},
  journal = {J. Algebra},
  volume  = {630},
  pages   = {248--273},
  year    = {2023},
  doi     = {10.1016/j.jalgebra.2023.03.046}
}

@incollection{Cos17,
  author    = {Coskun, Emre},
  title     = {A survey of {U}lrich bundles},
  booktitle = {Analytic and Algebraic Geometry},
  pages     = {85--106},
  publisher = {Springer},
  address   = {Singapore},
  year      = {2017},
  doi       = {10.1007/978-981-10-5648-2_6}
}

@book{CMP21,
  author    = {Costa, Laura and {Mir\'{o}-Roig}, Rosa Mar\'{i}a and {Pons-Llopis}, Joan},
  title     = {{U}lrich bundles---from commutative algebra to algebraic geometry},
  series    = {De Gruyter Studies in Mathematics},
  volume    = {77},
  publisher = {De Gruyter},
  address   = {Berlin},
  year      = {2021},
  doi       = {10.1515/9783110647686}
}

@article{CP91,
  author  = {Campana, Fr\'{e}d\'{e}ric and Peternell, Thomas},
  title   = {Projective manifolds whose tangent bundles are numerically effective},
  journal = {Math. Ann.},
  volume  = {289},
  number  = {1},
  pages   = {169--187},
  year    = {1991},
  doi     = {10.1007/BF01446566}
}

@article{DPS94,
  author  = {Demailly, Jean-Pierre and Peternell, Thomas and Schneider, Michael},
  title   = {Compact complex manifolds with numerically effective tangent bundles},
  journal = {J. Algebraic Geom.},
  volume  = {3},
  number  = {2},
  pages   = {295--345},
  year    = {1994}
}

@article{ESW03,
  author  = {Eisenbud, David and Schreyer, Frank-Olaf and Weyman, Jerzy},
  title   = {Resultants and {C}how forms via exterior syzygies},
  journal = {J. Amer. Math. Soc.},
  volume  = {16},
  number  = {3},
  pages   = {537--579},
  year    = {2003},
  doi     = {10.1090/S0894-0347-03-00423-5}
}

@book{Har77,
  author    = {Hartshorne, Robin},
  title     = {Algebraic Geometry},
  series    = {Graduate Texts in Mathematics},
  volume    = {52},
  publisher = {Springer-Verlag},
  address   = {New York-Heidelberg},
  year      = {1977},
  doi       = {10.1007/978-1-4757-3849-0}
}

@book{Huy05,
  author    = {Huybrechts, Daniel},
  title     = {Complex Geometry: An Introduction},
  series    = {Universitext},
  publisher = {Springer-Verlag},
  address   = {Berlin},
  year      = {2005},
  doi       = {10.1007/b137952}
}

@book{HL10,
  author    = {Huybrechts, Daniel and Lehn, Manfred},
  title     = {The Geometry of Moduli Spaces of Sheaves},
  series    = {Cambridge Mathematical Library},
  publisher = {Cambridge University Press},
  address   = {Cambridge},
  edition   = {Second},
  year      = {2010}
}

@article{KO73,
  author  = {Kobayashi, Shoshichi and Ochiai, Takushiro},
  title   = {Characterizations of complex projective spaces and hyperquadrics},
  journal = {J. Math. Kyoto Univ.},
  volume  = {13},
  number  = {1},
  pages   = {31--47},
  year    = {1973},
  doi     = {10.1215/kjm/1250523432}
}

@article{LR24,
  author  = {Lopez, Angelo Felice and Raychaudhury, Debaditya},
  title   = {On varieties with {U}lrich twisted tangent bundles},
  journal = {Ann. Mat. Pura Appl. (4)},
  volume  = {203},
  number  = {3},
  pages   = {1159--1193},
  year    = {2024},
  doi     = {10.1007/s10231-023-01397-w}
}

@article{Wat23,
  author  = {Watanabe, Kiwamu},
  title   = {Positivity of the exterior power of the tangent bundles},
  journal = {Proc. Japan Acad. Ser. A Math. Sci.},
  volume  = {99},
  number  = {10},
  pages   = {77--80},
  year    = {2023},
  doi     = {10.3792/pjaa.99.015}
}

@book{Kob87,
  author    = {Kobayashi, Shoshichi},
  title     = {Differential Geometry of Complex Vector Bundles},
  series    = {Publications of the Mathematical Society of Japan},
  volume    = {15},
  publisher = {Princeton University Press},
  address   = {Princeton, NJ},
  year      = {1987}
}

@article{BN18,
  author  = {Borisov, Lev and Nuer, Howard},
  title   = {{U}lrich bundles on {E}nriques surfaces},
  journal = {Int. Math. Res. Not. IMRN},
  number  = {13},
  pages   = {4171--4189},
  year    = {2018},
  doi     = {10.1093/imrn/rnw298}
}

@article{Cas24,
  author  = {Casnati, Gianfranco},
  title   = {Tangent, cotangent, normal and conormal bundles are almost never instanton bundles},
  journal = {Commun. Algebra},
  volume  = {52},
  number  = {2},
  pages   = {572--581},
  year    = {2024},
  doi     = {10.1080/00927872.2023.2245911}
}

@incollection{Lop25Normal,
  author    = {Lopez, Angelo Felice},
  title     = {On varieties with {U}lrich twisted normal bundles},
  booktitle = {Perspectives on Four Decades of Algebraic Geometry, Volume 2},
  series    = {Progress in Mathematics},
  volume    = {352},
  pages     = {1--12},
  publisher = {Birkh\"auser},
  address   = {Cham},
  year      = {2025},
  doi       = {10.1007/978-3-031-66234-8_1}
}

@article{ACLR24,
  author  = {Antonelli, Vincenzo and Casnati, Gianfranco and Lopez, Angelo Felice and Raychaudhury, Debaditya},
  title   = {On varieties with {U}lrich twisted conormal bundles},
  journal = {Proc. Amer. Math. Soc.},
  volume  = {152},
  number  = {11},
  pages   = {4645--4658},
  year    = {2024},
  doi     = {10.1090/proc/16986}
}

@misc{TLZ26,
  author        = {{Torres L\'opez}, Hugo and Zamora, Alexis G.},
  title         = {On the {U}lrichness of twisted syzygies and dual syzygies bundles},
  year          = {2026},
  eprint        = {2601.03406},
  archivePrefix = {arXiv},
  primaryClass  = {math.AG},
  note          = {arXiv:2601.03406 [math.AG]}
}

@misc{Sar26,
  author        = {Sarkar, Supravat},
  title         = {Extension of {U}lrich bundles},
  year          = {2026},
  eprint        = {2606.17429},
  archivePrefix = {arXiv},
  primaryClass  = {math.AG},
  note          = {arXiv:2606.17429 [math.AG]}
}
\end{document}